\documentclass[11pt,a4paper]{article}
\usepackage{ifpdf}

\newcommand{\Title}{Omnithermal Perfect Simulation for Multi-server Queues}
\usepackage[british]{babel}
\newcommand{\Date}{\today}

\usepackage{calc, xspace}
\usepackage{amsmath, amsthm, amsfonts, euscript, setspace}
\usepackage{algorithm}
\usepackage{algorithmic}

\usepackage[longnamesfirst]{natbib}
\setcitestyle{round}

\usepackage[T1]{fontenc}

\ifpdf
    \usepackage[dvipsnames]{xcolor}
\else
    \usepackage{color}
\fi
\usepackage{graphicx}

\definecolor{linkcolor}{named}{Maroon}
\definecolor{citecolor}{named}{PineGreen}
\definecolor{urlcolor}{named}{RoyalPurple}
\definecolor{okcolor}{named}{OliveGreen}
\definecolor{alertcolor}{named}{BrickRed}

\ifpdf
 \DeclareGraphicsExtensions{.png,.jpg}%
\else
 \DeclareGraphicsExtensions{.eps,.ps,.png,.jpg}%
\fi

\usepackage{a4wide}

\theoremstyle{plain}
 \newtheorem{thm}{Theorem}
 \newtheorem{lem}[thm]{Lemma}
 \newtheorem{prop}[thm]{Proposition}
 
\theoremstyle{definition}
 \newtheorem{defn}[thm]{Definition}
\theoremstyle{remark}
\newtheorem{rem}[thm]{Remark}
\newtheorem*{rems}{Remarks}
 \newtheorem{example}[thm]{Example}

\newcommand{\Expect}[1]{\operatorname{\mathbb{E}}\left[#1\right]}
\newcommand{\eps}{\varepsilon}

\usepackage{framed}			
\usepackage[small]{caption}
\usepackage[font=small]{subcaption}

{\begin{figure}[htbp]\begin{framed}}%
{\end{framed}\end{figure}}

\newcommand{\nonissue}[2]{}

 \usepackage[
 \ifpdf
   pdftex,
   pdfstartview=FitH,
 \fi
 ]{hyperref}
 \hypersetup{pdfstartview=FitH}
 \hypersetup{citecolor=citecolor}
 \hypersetup{linkcolor=linkcolor}
 \hypersetup{urlcolor=urlcolor}
 \hypersetup{colorlinks=true}

\begin{document}

 \title{\Title}
 \author{%
\href{http://www-users.york.ac.uk/~sbc502/}{Stephen B. Connor}
}
 \date{\Date}
 \maketitle


\begin{abstract}
A number of perfect simulation algorithms for multi-server First Come First Served queues have recently been developed. Those of \citet{Connor-Kendall-2015} 
and \citet{Blanchet-etal-2015} use dominated Coupling from the Past (domCFTP) to sample from the equilibrium distribution of the Kiefer-Wolfowitz workload vector for stable $M/G/c$ and $GI/GI/c$
queues respectively, using Random Assignment queues as dominating processes. In this note we answer a question posed by \citet{Connor-Kendall-2015}, by demonstrating how these algorithms may be modified in order to carry out domCFTP \emph{simultaneously}
for a range of values of $c$ (the number of servers).
\end{abstract}

 \begin{quotation}
 \end{quotation}
 
 \begin{quotation}\noindent
 \end{quotation}

\section{Introduction}
There have recently been a number of significant
advances in perfect simulation methods for multi-server
queues. \emph{Perfect simulation} algorithms are able to return an \emph{exact} sample from the  stationary distribution of an ergodic Markov chain (as opposed to an approximate sample, as may be obtained from e.g. Markov chain Monte Carlo) at the expense of a random run-time. 
The first practical algorithm in this line was the Coupling from the Past (CFTP) algorithm, which was conceived by \citet{ProppWilson-1996a} and used to sample from the exact equilibrium distribution of the critical Ising model on a finite lattice. The original CFTP algorithm has since been generalised in a number of ways, with the most significant of these, for the purposes of this paper, being dominated Coupling from the Past (domCFTP) \citep{Kendall-1998d,KendallMoeller-2000}. Importantly, domCFTP can be used for chains with unbounded state space; it relies on knowledge of a \emph{dominating process} for the chain of interest, which can be simulated both in equilibrium and in reverse-time.
\citet{Kendall-2005a} provides a nice introduction to perfect simulation algorithms, while a much more extensive account can be found in the recent book by \citet{Huber-2016}.

Stationary distributions arising from queueing systems involving multiple servers cannot in general be computed explicitly, and so practical methods for sampling from such distributions are of obvious interest.
\citet{Sigman-2011} pioneered the use of domCFTP for \emph{super-stable} $M/G/c$ queues with First
Come First Served (FCFS) discipline. (``Super-stable'' means that the
queue would remain stable even if all but one of the $c$ servers were
removed.) The limitation to super-stable queues is necessitated by
Sigman's use of a stable $M/G/1$ queue as dominating process in the domCFTP algorithm. \citet{Connor-Kendall-2015} subsequently showed how to generalise this idea to work for \emph{stable} $M/G/c$ queues, by using as dominating process an $M/G/c$ queue with Random Assignment (RA) discipline (under which the $c$ servers are independent). They describe two algorithms (outlined in Section~\ref{sec:domCFTP} below) and compare their efficiency; they show that their Algorithm 1, which requires waiting for the dominating process to empty, is significantly less efficient than Algorithm 2, which relies on the coalescence of sandwiching processes (in common with many other domCFTP algorithms).

\citet{Blanchet2018} were the first authors to show how to perform perfect simulation for multi-server queues with general inter-arrival time and service time distributions (i.e. relaxing the assumption of exponential inter-arrival times). Rather than using a random assignment queue as dominating process, they make use of a so-called ``vacation system''. (This idea is also employed by \citet{Blanchet2019} to sample from the equilibrium of a generalized Jackson network of single-server queues.) However, \citet{Blanchet-etal-2015} have since demonstrated how to make the random assignment dominating process work in this setting. The hard part here is working out how to simulate the dominating process in reverse-time; with renewal arrivals, as opposed to Poisson, the $c$ servers in the RA model are no longer independent. These pieces of work all serve to demonstrate that perfect simulation is a practical and efficient method for simulating from a wide class of multi-server queueing systems.

\citet{Connor-Kendall-2015} ask a very natural question: is it possible to carry out dominated CFTP \emph{simultaneously} for \(M/G/c\) queues with a range of \(c\), the number of servers? The authors refer to this as ``omnithermal dominated CFTP'', borrowing a term used to describe \citet{Grimmett-1995}'s coupling of random-cluster processes
for all values of a specific parameter, and applied to CFTP in \cite{ProppWilson-1996a}. The potentially difficult issue in the queueing context is that of detecting a time at which we can be sure that the appropriate sandwiching processes will coalesce for all \(c\) in the range being considered. In this paper we show how such coalescence may be detected with the aid of a simple criterion that uses information about the sandwiching processes only for the queue with the fewest servers. 

The outline of the paper is as follows. In Section~\ref{sec:domCFTP} we recall the definition of the Kiefer-Wolfowitz workload process associated to a multi-server FCFS queue, and then sketch the two perfect simulation algorithms of \cite{Connor-Kendall-2015}. In Section~\ref{sec:omnithermal} we present a natural partial order between Kiefer-Wolfowitz vectors of different lengths, and subsequently use this to determine a condition which ensures that the termination time of \citet{Connor-Kendall-2015}'s Algorithm 2 is monotonic in the number of servers $c$. In Section~\ref{sec:sims} we use this condition to produce an Omnithermal Algorithm, and briefly report on the results of applying this to some $M/M/c$ queues. Finally, in Section~\ref{sec:conclusions} we indicate how our results may be used to perform omnithermal perfect simulation for queues with general renewal input, or in the situation where we are interested in scaling the distribution of service durations, rather than changing the number of servers. The question of sampling from the equilibrium distribution of \emph{controlled} or \emph{adaptive} systems, e.g. where the number of servers is allowed to change over time, is also briefly considered.

\section{Dominated CFTP for $M/G/c$ queues}\label{sec:domCFTP}

Consider a general $\cdot/\cdot/c$ FCFS queue. We denote the  Kiefer-Wolfowitz workload vector \citep{KieferWolfowitz-1955} at time $t\geq 0$ by ${\mathbf V}(t) = (V(1,t),V(2,t),\dots, V(c,t))$, where $V(1,t)\leq V(2,t) \leq\dots$. The entries of ${\mathbf V}(t)$ represent the ordered amounts of residual work in the system for the $c$ servers at time $t$, bearing in mind the FCFS queueing discipline. Customer \(n\) arrives at time $t_n$ (for \(0\leq t_1\leq t_2\leq \ldots\)), with inter-arrival times denoted by $T_n = t_{n+1}-t_n$ (with $t_0 = 0$). Customer \(n\) brings with it a service duration $S_n$.
Observing $\mathbf V$ just before arrival of the \(n^\text{th}\) customer (but definitely after the arrival of the \((n-1)^\text{th}\) customer) generates a process $\mathbf W_n = (W_n(1),W_n(2),\dots, W_n(c))$: in the case \(t_{n-1}<t_n\) we have
$\mathbf W_n= \mathbf V(t_n-)$. This satisfies the well-known recursion
\begin{equation*}\label{eqn:KW-recursion}
\mathbf W_{n+1} \quad=\quad R(\mathbf W_n + S_n\mathbf e - T_n\mathbf f)^+ ,\quad \text{ for }n\geq 0\,, 
\end{equation*}
where $\mathbf e=(1,0,0,\dots,0)$, $\mathbf f=(1,1,\dots,1)$, $R$ places the coordinates of a vector in increasing order, and $^+$ replaces negative coordinates of a vector by zeros (see Chapter 12 of~\cite{Asmussen-2003}). In words, $\mathbf W_{n+1}$ is obtained from $\mathbf W_{n}$ by performing the following sequence of operations:
\begin{enumerate}
	\item add the new workload $S_n$ to the first coordinate (the server currently with least residual work)
	\item subtract $T_n$ from each coordinate (since each server completes work at unit rate between arrival times)
	\item reorder the coordinates of the resulting vector in increasing order
	\item replace any negative coordinates by zeros.
\end{enumerate}
Note in particular that $W_n(1)$ represents the time that the $n^\text{th}$ customer must wait before commencing service.

For simplicity of exposition we shall primarily discuss $M/G/c$ queues in what follows (i.e. inter-arrival times are exponential). However, our method for performing omnithermal perfect simulation for these queues applies equally well to $GI/GI/c$ queues using an algorithm of \citet{Blanchet-etal-2015}, as will be observed in Section~\ref{sec:conclusions}. Let the arrival rate be $\lambda>0$, and let service durations $S_n$ be i.i.d. with mean $1/\mu$ and $\Expect{S^2}<\infty$. (As explained in \cite{Connor-Kendall-2015}, this second moment condition is required in order to guarantee finite mean run-time of their perfect simulation algorithms. In the case of $GI/GI/c$ queues a little more is required, namely that the inter-arrival times and service durations both have a finite $2+\eps$ moment, for some $\eps>0$~\citep{Blanchet-etal-2015}.) Write $\rho=\lambda/(c\mu)$: the queue is stable if and only if $\rho<1$, in which case it is known that $\mathbf{W}_n$ converges as $n\to\infty$ to an equilibrium distribution, and so we restrict attention to this scenario.	

\citet{Connor-Kendall-2015} propose two domCFTP algorithms for sampling from the equilibrium distribution of the Kiefer-Wolfowitz workload vector for a stable $M/G/c$ queue $X$. Both of these algorithms use as dominating process an $M/G/c$ queue $Y$ with Random Assignment service discipline. That is, customers in $Y$ are allocated upon arrival to a uniformly chosen server; this renders the $c$ servers independent, which allows us to easily simulate a stationary version of the dominating process in reverse-time, as required by domCFTP. It is possible to arrange for $X$ to be path-wise dominated by $Y$ as long as the two queues are coupled by assigning service durations in order of \emph{initiation of service}. (Under FCFS customers initiate service in the same order in which they arrive, but this is typically not the case for other service disciplines.) The precise statement of this domination can be found in \cite{Connor-Kendall-2015}, an abridged version of which is reproduced here for convenience.

\begin{thm}[Theorem 3.3 of \cite{Connor-Kendall-2015}]\label{thm:CK3.3}
	Consider a \(c\)-server queueing system viewed as a function of (a) the sequence of arrival times \(0\leq t_1\leq t_2\leq t_3\leq \ldots\) and (b) the sequence of service durations \(S_1, S_2, S_3, \ldots\) assigned in order of initiation of service. Consider the following different allocation rules, in some cases varying over time:  
	\begin{enumerate}
		\item $\cdot/\cdot/c\;[RA]$;
		\item $\cdot/\cdot/c\; [RA]$ until a specified non-random time $T$, then switching to $\cdot/\cdot/c\; [FCFS]$;
		\item $\cdot/\cdot/c\; [RA]$ until a specified non-random time \(T'\), $0\leq T'\leq T$, then switching to \\$\cdot/\cdot/c\; [FCFS]$;
		\item $\cdot/\cdot/c\; [FCFS]$;
	\end{enumerate}
	Then case \(k\) dominates case \(k+1\) (for $k=1,2,3$), in the sense that the \(m^\text{th}\) initiation of service in case \(k+1\) occurs no later than the \(m^\text{th}\)
	initiation of service in case \(k\), and the \(m^\text{th}\) departure in case \(k+1\) occurs no later than the \(m^\text{th}\)
	departure in case \(k\). Moreover, for all times $t\geq T'$ the Kiefer-Wolfowitz workload vector for case 3 dominates (coordinate-by-coordinate) that of case 4, with similar domination holding for cases 2 and 3 for all $t\geq T$.
\end{thm}

We can now summarise the two domCFTP algorithms of \cite{Connor-Kendall-2015}.
\begin{algorithm}\caption{}
	\label{alg1}
	\begin{algorithmic}[1]
		\STATE  Construct a stationary $M/G/c\,[RA]$ process backwards in time until it empties at some time $T^*<0$;
		\STATE Use this to create a forwards in time trajectory of an $M/G/c\,[RA]$ queue $Y$ started from empty at time $T^*$; 
		\STATE Use the sequences of arrival times and service durations in $Y$ to construct an $M/G/c\,[FCFS]$ queue $X=\{X_t\,:\,T^*\le t\le 0\}$ that is dominated by $Y$ over $[T^*,0]$, and return $X_0$.
	\end{algorithmic}
\end{algorithm}

Steps 1 and 2 of the algorithm are accomplished as follows (see \cite{Connor-Kendall-2015} for further details.) We first simulate the path of a collection of $c$ stationary $M/G/1$ queues, each of which has arrival rate $\lambda/c$, and which complete work using the Processor Sharing discipline (whereby all jobs are served simultaneously, at a rate depending upon the number of jobs present). We perform this over the time period $[0,\hat T^*]$ where $\hat T^*\ge 0$  is the first time at which all $c$ servers are simultaneously empty, and record the set of \emph{departure times} $0\le \hat t_1\le \hat t_2\le \dots\le \hat t_k = \hat T^*$ and associated service durations $S_1,\dots,S_k$. The $M/G/c\,[RA]$ queue $Y$ is then started from empty at time $T^* = -\hat T^*$ and fed the sequence of arrival times/service durations $(-\hat t_k,S_k),\dots,(-\hat t_1,S_1)$. 

To carry out Step 3, we reorder the set of service durations according to the corresponding \emph{initiation of service} in $Y$. We denote this reordered list by $(S'_1,\dots,S'_k)$: if $J^Y_i$ is the time of initiation of service $S'_i$ in $Y$,	then $T^* = J^Y_1 \leq  J^Y_2 \leq \dots \leq J^Y_k$. Finally, the $M/G/c\,[FCFS]$ queue $X$ is started from empty at time $T^*$ and fed the sequence of arrival times/service durations $(-\hat t_k,S'_1),\dots,(-\hat t_1,S'_k)$. Since both $Y$ and $X$ see the same sequence of arrival times over $[T^*,0]$, and use a common sequence of service durations assigned in order of initiation of service, the domination argument of Theorem~\ref{thm:CK3.3} holds; a standard domCFTP argument then shows that $X_0$ is a draw from the required equilibrium distribution.

\begin{algorithm}\caption{}
	\label{alg2}
	\begin{algorithmic}[1]
		\STATE  Fix a \emph{backoff} (or \emph{inspection}) time $T<0$, and construct a path of the stationary $M/G/c\,[RA]$ queue $Y$ over the time period $[T,0]$; 
		\STATE Construct \emph{sandwiching processes} $L^c = \{L^c_t\,:\,T\le t\le 0\}$ and $U^c= \{U^c_t\,:\,T\le t\le 0\}$ over $[T,0]$ as follows: 
		\begin{enumerate}
			\item $L^c_T$ is empty, while $U^c_T$ is instantiated using the same residual workloads present in $Y_T$;
			\item over $(T,0]$, $L^c$ and $U^c$ both evolve as Kiefer-Wolfowitz vectors of $M/G/c\,[FCFS]$ queues, using the same sequences of arrival times and service durations as $Y$ (once again ordered by initiation of service); 
		\end{enumerate}
		\STATE  Check for coalescence: if $L^c_0 = U^c_0$ return this value; else set $T\leftarrow2T$ and go to Step 1.
	\end{algorithmic}
\end{algorithm}
Connor and Kendall~\cite{Connor-Kendall-2015} provide more details for each of the steps outlined above, and demonstrate that Algorithm~\ref{alg2}, although more complicated to describe, is in general significantly faster than Algorithm~\ref{alg1}. 


\section{Omnithermal perfect simulation}\label{sec:omnithermal}

In this section we consider the following question: is it possible to adapt the domCFTP algorithms outlined in Section~\ref{sec:domCFTP} in order to \emph{simultaneously} sample from the equilibrium of $M/G/(c+m)$ queues for all $m\ge 0$? As pointed out in \cite{Connor-Kendall-2015}, it is straightforward to  accomplish this using Algorithm 1: once an emptying time \(T^*\) has been established for the \(M/G/c\) queue, then any $M/G/(c+m)$ queue may be started from empty at this time and run over $[T^*,0]$ using the same arrival times and service durations; a simple workload domination argument shows that its value at time 0 will form a single perfect sample from the required equilibrium distribution. However, given the significantly faster run-time of Algorithm~\ref{alg2}, a far more interesting question is whether or not one can produce a comparably efficient omnithermal domCFTP algorithm using sandwiching processes.

Suppose that we have implemented Algorithm~\ref{alg2}, and have obtained one equilibrium sample for the $M/G/c$ queue. That is, we have established some backoff time $T<0$, along with sequences of arrival times and service durations, such that $L^c_0=U^c_0$. Our first observation is the following: suppose that we use these sequences to produce new FCFS processes $L^{c+m}$ and $U^{c+m}$ over $[T,0]$ in the manner described in Step 2 of Algorithm~\ref{alg2}. More explicitly, $L^{c+m}_T$ is empty, and $U^{c+m}_T$ is constructed by feeding in all of the residual workloads present in $Y_T$, in order of their initiation of service in $Y$. (In particular, this means that if there are more than $c$ jobs present in $Y_T$ then more than $c$ coordinates of $U^{c+m}_T$ will be non-zero.) $L^{c+m}$ and $U^{c+m}$ are then fed the same sequences of arrival times and service durations as $L^c$ (and $U^c$) over $(T,0]$. The results of doing this are that the workload vector for $U^{c+m}_t$ will dominate (coordinate-by-coordinate) that of $L^{c+m}_t$ for all $t\in[T,0]$, and if $L^{c+m}_0=U^{c+m}_0$ then this value will be a perfect draw from the equilibrium of the $M/G/(c+m)$ queue, as required. This follows from Theorem~\ref{thm:CK3.3}: due to the way in which it is instantiated, $U^{c+m}$ is a queueing system that changes from $M/G/c\,[RA]$ to $M/G/(c+m)\,[FCFS]$ at time $T$. But the former of these can be thought of as an $M/G/(c+m)$ system with a random allocation rule which uniformly distributes jobs amongst only a fixed $c$ of the $(c+m)$ servers; since this is less efficient than the FCFS discipline, the proof of Theorem 3.3 in \cite{Connor-Kendall-2015} holds with this slightly modified setup.

This observation implies that, given the arrival times and service durations used in Algorithm~\ref{alg2} with $c$ servers, we could just construct sandwiching processes $L^{c+m}$ and $U^{c+m}$ over $[T,0]$ and see whether they coalesce. If they do, then we have obtained a sample from the required distribution; if not, then we need to extend the dominating process $Y$ for this sample further into the past (setting $T\leftarrow2T$), and then check again for coalescence. But this is not as clean as we would like: as will be shown in the next section, coalescence of $L^{c}$ and $U^{c}$ over $[T,0]$ does not imply coalescence of $L^{c+m}$ and $U^{c+m}$ over the same interval for all $m>0$. Thus it is possible that the extent to which any single sample path of $Y$ needs to be extended will vary with the value of $m$. Assuming that we want to obtain samples for a range of values of $m$, this method is therefore rather inefficient. Ideally we would like to use Algorithm~\ref{alg2} to produce a sample for the $M/G/c$ queue, and then re-use the path of $Y$ from this run of the algorithm in order to draw from the equilibrium of $M/G/(c+m)$ for any $m>0$.

\subsection{Comparing queues with different numbers of servers}

Suppose that we have two FCFS queues,
each seeing the same set of arrival times and associated service
durations. We first of all need to show that the workload vector with
fewer servers dominates that of the other, with respect to a certain natural
partial order.

\begin{defn}\label{def:parial_order}
	For $V^c\in\mathbb{R}^c$ and $V^{c+m}\in\mathbb{R}^{c+m}$, we
	write $V^{c+m}\preceq V^c$ if and only if
	\[ V^{c+m}(k+m) \le V^c(k)\,, \quad k=1,\dots,c \,. \]
\end{defn}
Thus if $V^c$ and $V^{c+m}$ are workload vectors, $V^{c+m}\preceq V^c$
if and only if each of the $c$ busiest servers in $V^{c+m}$ has no more
work remaining than the corresponding server in $V^c$.

\begin{prop}\label{prop:dom_different_c}
	Let $V^c$ and $V^{c+m}$ be Kiefer-Wolfowitz workload vectors for an
	$M/G/c$ and an $M/G/(c+m)$ FCFS queue respectively. Suppose that
	$V^{c+m}_0\preceq V^c_0$ and that each queue sees the same set of
	arrival times and associated service durations. Then $V^{c+m}_t\preceq
	V^c_t$ for all $t\ge 0$.
\end{prop}
\begin{proof}
	It is clear that the ordering between $V^c$ and $V^{c+m}$ will hold
	until the first arrival time, $\tau$. Furthermore, once we show that
	$V^{c+m}_{\tau}\preceq V^c_{\tau}$ the result will follow simply by
	induction.
	
	Let $S$ denote the service duration attached to the arrival at time
	$\tau$. Recall that the effect of this arrival is that $S$ is added
	to any outstanding work at the first (least busy) coordinate in
	$V^c$ and $V^{c+m}$, and then the resulting vectors are each reordered in
	increasing order. Suppose that after this reordering has taken place, the coordinate with value  $V^c_{\tau-}(1)+S$ (the amount of work now at the server to which the arrival at $\tau$ was allocated) is located in
	position $i^c$ of $V^c_{\tau}$, etc. Note that the result of the reordering is precisely the following:
	\begin{equation}\label{eqn:change_V}
	V^c_{\tau}(k) = \begin{cases}
	V^c_{\tau-}(k+1) & \quad k<i^c \\
	V^c_{\tau-}(1) + S & \quad k=i^c	 \\
	V^c_{\tau-}(k) & \quad k>i^c\,, \\
	\end{cases}
	\quad\text{and}\quad 
	V^{c+m}_{\tau}(k) = \begin{cases}
	V^{c+m}_{\tau-}(k+1) & \quad k<i^{c+m} \\
	V^{c+m}_{\tau-}(1) + S & \quad k=i^{c+m}	 \\
	V^{c+m}_{\tau-}(k) & \quad k>i^{c+m}\,.
	\end{cases}
	\end{equation}
	
	If $i^{c+m}\le m$ then the result is trivial (since the last $c$ coordinates of $V^{c+m}$ are unchanged by the arrival at time $\tau$, and so the ordering between the vectors at time $\tau-$ is clearly maintained). So suppose that $i^{c+m}>m$. Then for
	$k<\min\{i^c,i^{c+m}-m\}$ we have
	\[ V^{c+m}_{\tau}(k+m)= V^{c+m}_{\tau-}(k+m+1) \le V^c_{\tau-}(k+1) =
	V^c_{\tau}(k) \,. \] (Here both of the equalities follow from \eqref{eqn:change_V}, and the inequality from the assumption that $V^{c+m}_{\tau-}\preceq V^c_{\tau-}$.) Analogously, for $k>\max\{i^c,i^{c+m}-m\}$ we have
	\[ V^{c+m}_{\tau}(k+m)= V^{c+m}_{\tau-}(k+m) \le V^c_{\tau-}(k) =
	V^c_{\tau}(k) \,. \]

	For the remaining coordinates there are now two cases to consider, depending on which of $i^c$ and $i^{c+m}-m$ is larger.
	\begin{description}
		\item Case 1: $i^c\le i^{c+m}-m$. Then for $k=i^c,\dots,i^{c+m}-m$:
		\[ V^{c+m}_{\tau}(k+m) \le V^{c+m}_{\tau}(i^{c+m}) = V^{c+m}_{\tau-}(1)+S \le V^c_{\tau-}(1)+S = V^c_{\tau}(i^c) \le V^c_{\tau}(k)\,. \]
		Here the first and last inequalities hold since the coordinates of the workload vectors at time $\tau$ are arranged in increasing order; the middle inequality follows from $V^{c+m}_{\tau-}\preceq V^c_{\tau-}$, and the equalities follow from \eqref{eqn:change_V}\\
		
		\item Case 2: $i^{c+m}-m< i^c$. (Recall that we are already supposing that $i^{c+m}>m$, and so $i^c>1$ here.)
		For $k=i^{c+m}-m,\dots,i^c-1$, using similar arguments as for Case 1, we see that
		\[ V^{c+m}_{\tau}(k+m) \le V^{c+m}_{\tau}(k+m+1) = V^{c+m}_{\tau-}(k+m+1) \le V^c_{\tau-}(k+1) = V^c_{\tau}(k) \,. \]
		The proof is completed by observing that when $k=i^c$,
		\[ V^{c+m}_{\tau}(k+m) = V^{c+m}_{\tau-}(i^c+m) \le V^c_{\tau-}(i^c) = V^c_{\tau}(i^c-1)\le V^c_{\tau}(i^c)\,.\qedhere \]
	\end{description}
\end{proof}

\subsection{Coalescence}

Suppose once again that we have used Algorithm~\ref{alg2} to obtain a single perfect
sample from the $M/G/c$ queue: this yields a backoff time $T<0$ and a
sequence of arrival times and associated service durations such that the
sandwiching processes $U^c$ and $L^c$ coalesce over $[T,0]$. Define
$D^c$ to be the non-negative vector-valued process given by the coordinate-wise
difference between $U^c$ and $L^c$:
\[ D^c_t = U^c_t - L^c_t\,, \quad T\le t \le 0 \,.\]

Let $T^c$ be the coalescence time for this realisation:
\[ T^c = \inf\{t>T\,:\, D^c_t = 0\} < 0 \,. \] We shall write
$|L^c_t|$ for the number of customers in $L^c_t$, and $\mathcal A^c_t$
for the set of coordinates where $U^c_t$ and $L^c_t$ \emph{agree}:
\[ \mathcal A^c_t = \{k\,:\, D^c_t(k)=0,\, 1\le k \le c\} \,. \]

We are interested in the question of whether coalescence of $U^c$ and $L^c$ implies coalescence of $U^{c+m}$ and $L^{c+m}$ (instantiated at time $T$ as described in Section~\ref{sec:domCFTP}) over the same period. The following example shows that this is \emph{not} guaranteed.

\begin{example}\label{ex:fails}
	Consider sandwiching processes for two and three server systems (i.e. $c=2$ and $m=1$), as described above. 
	Suppose that $U^2$ and $U^3$ are both instantiated at time $T=-4$ with a single service duration of length 1, and that these queues proceed to see pairs of arrival times and services $(t,S)$ as follows: $(-3.9, 1.2)$,
	$(-3.7, 1.8)$, $(-3.2,5)$. The evolution of these processes viewed at arrival times is as follows:
	
	\begin{center}
		\begin{tabular}{c|cccc}
			&	$t_0 = -4$ 	&		$t_1 = -3.9$	&		$t_2 = -3.7$		&	$t_3 = -3.2$ \\
			\hline
			$U^2$	&	$(0.0,\,1.0)$		&	$(0.9,\,1.2)$ & $(1.0,\,2.5)$ & $(2.0,\,5.5)$ \\
			$L^2$	&	$(0.0,\,0.0)$ 		&	$(0.0,\,1.2)$	&	$(1.0,\,1.8)$		&	$(1.3,\,5.5)$
		\end{tabular} 
	\end{center}

	If there are no further arrivals within the next two units of time, we see that $U^2$ and $L^2$ will coalesce at time $T^2 = -1.2$ (since it will take two more units of time for their first coordinates to agree, and their second coordinates are already matched).
	
	However, feeding the same sequence of arrival times/services to $U^3$ and $L^3$, we see that they will \emph{not} coalesce before time $T^2$:
	
	\begin{center}
		\begin{tabular}{c|cccc}
			&	$t_0 = -4$ 	&		$t_1 = -3.9$	&		$t_2 = -3.7$		&	$t_3 = -3.2$\\
			\hline
			$U^3$	&	$(0.0,\,0.0,\,1.0)$		&	$(0.0,\,0.9,\,1.2)$ & $(0.7,\,1.0,\,1.8)$ &  $(0.5,\,1.3,\,5.2)$ \\
			$L^3$	&	$(0.0,\,0.0,\,0.0)$ 		&	$(0.0,\,0.0,\,1.2)$		&	$(0.0,\,1.0,\,1.8)$		&	$(0.5,\,1.3,\,5.0)$
		\end{tabular} 
	\end{center}
	
	Furthermore, if we were to consider sandwiching processes for a four-server system, these \emph{would} coalesce by time $T^2$ using the above sequence of arrivals:
	
	\begin{center}
		\begin{tabular}{c|cccc}
			&	$t_0 = -4$ 	&		$t_1 = -3.9$	&		$t_2 = -3.7$		&	$t_3 = -3.2$ \\
			\hline
			$U^4$	&	$(0.0,\,0.0,\,0.0,\,1.0)$		&	$(0.0,\,0.0,\,0.9,\,1.2)$ & $(0.0,\,0.7,\,1.0,\,1.8)$  & $(0.2,\,0.5,\,1.3,\,5.0)$ \\
			$L^4$	&	$(0.0,\,0.0,\,0.0,\,0.0)$ 		&	$(0.0,\,0.0,\,0.0,\,1.2)$		&	$(0.0,\,0.0,\,1.0,\,1.8)$		&	$(0.0,\,0.5,\,1.3,\,5.0)$
		\end{tabular} 
	\end{center}
\end{example}

A simple, and intuitively obvious, condition which guarantees that the sandwiching processes $L^{c+m}$ and $U^{c+m}$ \emph{will}  coalesce by time $T^c$ is that no customer arriving at the lower process $L^c$ during the period $[T,T^c]$ has to wait to commence service:
\begin{prop}\label{prop:L_never_full}
	If $|L^c_t|\le c$ for all $t\in[T,T^c]$ then $D^{c+m}_{T^c}=0$ (and so $T^{c+m}\le T^c$) for any $m\ge 0$.
\end{prop}
\begin{proof}
	Since no server in $L^c$ ever has more than one customer to deal with at any moment, the same is true for $L^{c+m}$, and so $L^{c+m}_t(k+m) = L^c_t(k)$ for all
	$1\le k\le c$ and for all $t\in[T,T^c]$. Then by the domination established in
	Proposition~\ref{prop:dom_different_c}, and the fact that $U^c_{T^c}=L^c_{T_c}$,
	\[ U^{c+m}_{T^c}(k+m) \le U^c_{T^c}(k) = L^c_{T^c}(k) = L^{c+m}_{T^c}(k+m)\,, \]
	for all $1\le k\le c$, and so the final $c$ coordinates of $L^{c+m}_{T^c}$ and $U^{c+m}_{T^c}$ must agree.
	
	Moreover, coalescence of $U^c$ and $L^c$ implies that there must exist an empty server in both of these processes at time $T^c$ (see \cite{Connor-Kendall-2015}); i.e. $U^c_{T^c}(1) =
	L^c_{T^c}(1) = 0$. Since $U^{c+m}_T\preceq U^c_T$, Proposition~\ref{prop:dom_different_c} ensures that $U^{c+m}_{T^c}\preceq U^c_{T^c}$, and so the first $m$ coordinates of $U^{c+m}_{T^c}$, and of $L^{c+m}_{T^c}$, must all equal zero. Thus $L^{c+m}_{T^c} = U^{c+m}_{T^c}$, as required.
\end{proof}

The condition of Proposition~\ref{prop:L_never_full} is rather strong, and can in fact be weakened, as we now show. 

\begin{thm}\label{thm:main}
	Suppose that any arrival time $\tau\in[T,T^c]$ satisfying $L^c_{\tau-} (1) = U^c_{\tau-} (1)$ (equivalently, $1\in\mathcal A^c_{\tau-}$) also satisfies $U^c_{\tau-}(1) = 0$. Then $T^{c+m}\le T^c$ for any $m\ge 0$.
\end{thm}

In other words, coalescence of $U^{c+m}$ and $L^{c+m}$ is guaranteed by time $T^c$ as long as the following holds: whenever an arriving job finds the same amount of residual work at its allocated servers in $U^c$ and $L^c$, that's precisely because both of those servers are \emph{idle} at that moment.

\begin{rem}
	The condition of Proposition~\ref{prop:L_never_full} is stronger than that of Theorem~\ref{thm:main}. To see this, suppose that $|L^c_t|\le c$ for all $t\in[T,T^c]$. If at some
	arrival time $\tau\in[T,T^c]$ we have $D^c_{\tau-} (1) = 0$ but $U^c_{\tau-}(1) > 0$, then there must be at least $c$ customers in
	$L^c_{\tau-}$ (since $L^c_{\tau-}(1)=U^c_{\tau-}(1)>0$). But then
	the customer arriving at time $\tau$ would force $|L^c_{\tau}|=c+1$,
	which would break our initial assumption. Therefore if $U^c_{\tau-} (1)-L^c_{\tau-} (1) = D^c_{\tau-} (1) = 0$ it
	must be the case that $U^c_{\tau-}(1) = 0$. 
	
	On the other hand, consider a two server system in which $U^2$ is instantiated at time $T=-2$ with a single service duration of length 1, and which sees  pairs of arrival times and services $(t,S)$ as follows: $(-1.9, 1.2)$,
	$(-1.7, 0.6)$, $(-1.6,0.2)$. It is simple to check that if there are no further arrivals, $U^2$ and $L^2$ will coalesce at time $-0.4$. Furthermore, the only arrival time at which $D^2_{\tau-}(1) = 0$ is $\tau=-1.9$, with $U^2_{\tau-}(1) = 0$; thus the condition of Theorem~\ref{thm:main} is satisfied by this example. However, the condition of Proposition~\ref{prop:L_never_full} clearly fails, since $|L^2_t|=3$ for $t\in[-1.6,-1.1)$.
	
	(Note that in Example~\ref{ex:fails} the condition of Theorem~\ref{thm:main} clearly fails for arrival time $t_3$.)
\end{rem}

The key to proving Theorem~\ref{thm:main} is to consider the time until coalescence of the sandwiching processes $U^c$ and $L^c$ when viewed at time $t\ge T$, i.e. the time taken for $U^c_t$ to clear all work in coordinates which disagree with those in $L^c_t$. Let us write $C^c_t$ for this quantity:
\begin{equation}\label{eqn:C_defn}
C^c_t = \max_{k\notin\mathcal A^c_t}U^c_t(k)  =  U^c_t(n^c_t) \,,
\end{equation}
where we define $n^c_t = \max\{1\le k\le c\,:\, k\notin \mathcal A^c_t\}$.

It is clear that the process $C^c = \{C^c_t\,:\, T\le t\}$ decreases deterministically at unit rate until it either hits zero (at which point $L^c$ and $U^c$ coalesce)  or a new customer arrives. Consider then what happens to $C^c$ if there is an arrival at time $\tau$ with associated service duration $S$. Let $k_U$ and $k_L$ be the coordinates satisfying $U^c_\tau(k_U) = 
U^c_{\tau-}(1)+S$ and $L^c_\tau(k_L) = L^c_{\tau-}(1)+S$. That is, the arriving job
gets allocated to the server with the least work in each of $U^c_{\tau-}$ and $L^c_{\tau-}$, and then when the workload vectors are reordered this job finds itself in position $k_U$ in $U^c_\tau$ and $k_L$ in $L^c_\tau$. To be explicit
\[ k_U = \min\{k\,:  U^c_{\tau-}(1)+S \le U^c_{\tau-}(k+1)\,, 1\le k < c  \} \,,\]
with $k_U = c$ if the minimum above is taken over the empty set. Note that this convention -- that the new job is placed at the \emph{lowest} coordinate possible, after reordering, in $U^c_\tau$ -- allows us to deal with the possibility that $U^c_{\tau-}(1) +S= U^c_{\tau-}(k+1)$ for some $1\le k< c$, which would result in the vector $U^c_\tau$ having two matching but non-zero entries. (When arrivals are Poisson this possibility occurs with probability zero, of course, in which case this convention is somewhat unnecessary.) In particular, this implies that 
\begin{equation}\label{eqn:convention}
U^c_\tau(k_U) > U^c_{\tau-}(k_U) \,.
\end{equation}

There are two cases to consider when assessing the impact of an arrival on $C^c$, depending on whether or not the servers with least workload in $U^c_{\tau-}$ and $L^c_{\tau-}$ are in agreement.

\begin{description}
	\item Case 1: $1\in \mathcal A^c_{\tau-}$
	\begin{enumerate}
		\item[(i)] Suppose first that $k_U\ge n^c_{\tau-}$. Since $U^c_{\tau-}(k) =
		L^c_{\tau-}(k)$ for all $k>n^c_{\tau-}$, it must be the case that $k_L = k_U\in \mathcal A^c_\tau$. So $n^c_\tau = n^c_{\tau-}-1$ and
		$C^c_\tau = U^c_\tau(n^c_\tau) = U^c_{\tau-}(n^c_{\tau-}) = C^c_{\tau-}$.
		\item[(ii)] Alternatively, if $k_U<n^c_{\tau-}$ then $n^c_\tau = n^c_{\tau-}$,
		and so $C^c_\tau = C^c_{\tau-}$ once again.
	\end{enumerate}
	Thus there is no change to $C^c$ if the arriving customer finds $U^c_{\tau-}(1) = L^c_{\tau-}(1)$. \\
	\item Case 2: $1\notin \mathcal A^c_{\tau-}$
	\begin{enumerate}
		\item[(i)] Suppose that $k_U\ge n^c_{\tau-}$. Since $U^c_{\tau-}(k) =
		L^c_{\tau-}(k)$ for all $k>n^c_{\tau-}$, it must be the case that $k_L \le
		k_U$. We claim that $k_U\notin\mathcal A^c_\tau$, and so $n^c_\tau= k_U$; it then follows that $C^c_\tau = U^c_\tau(n^c_\tau) = U^c_\tau(k_U) =  U^c_{\tau-}(1) +S>
		C^c_{\tau-}$.
		
		To see that $n^c_\tau= k_U$, we need to show that $L^c_{\tau}(k_U)<U^c_\tau(k_U)$. Notice that 
		\[  L^c_\tau(k_U) = \max\{L^c_{\tau-}(k_U),\,  L^c_{\tau-}(1)+S\} \quad\text{and}\quad U^c_\tau(k_U) = U^c_{\tau-}(1)+S   \,. \]
		Clearly $L^c_{\tau-}(1)+S<U^c_{\tau-}(1)+S$ (since $1\notin \mathcal A^c_{\tau-}$). Furthermore, 
		$L^c_{\tau-}(k_U) \le U^c_{\tau-}(k_U) < U^c_\tau(k_U)$ thanks to \eqref{eqn:convention}.
		
		\item[(ii)] Alternatively, if $k_U<n^c_{\tau-}$ then $k_L\le n^c_{\tau-}$
		also, and so $n^c_\tau = n^c_{\tau-}$. Thus $C^c_\tau = C^c_{\tau-}$.
	\end{enumerate}
	Thus when $1\notin \mathcal A^c_{\tau-}$, $C^c_\tau = \max\{U^c_{\tau-}(n^c_{\tau-}),\,  U^c_{\tau-}(1) +S\}$.
\end{description}

In summary, we see that $C^c_t$ increases only at arrival times $\tau$ for which 
$1\notin\mathcal A^c_{\tau-}$ and $k_U\ge n^c_{\tau-}$. That is,
\begin{equation}\label{eqn:change_in_C}
C^c_\tau = \begin{cases}
C^c_{\tau-} & \quad\text{if $1\in \mathcal A^c_{\tau-}$}\\
\max\{C^c_{\tau-},\,  U^c_{\tau-}(1)+S \} &\quad\text{if $1\notin \mathcal A^c_{\tau-}$}\,.
\end{cases}
\end{equation}

The next result is key to the proof of Theorem~\ref{thm:main}: it shows that, under the same assumption as the theorem, the time to coalescence with $c+m$ servers is dominated by the time to coalescence with $c$ servers.
\begin{lem}\label{lem:monotone}
	Fix some $m\ge 0$, and suppose that the following two conditions both hold at arrival time $\tau$:
	\begin{itemize}
		\item $C^{c+m}_{\tau-} \le C^c_{\tau-}$
		\item if $1\in\mathcal A^c_{\tau-}$	then $L^c_{\tau-} (1) = U^c_{\tau-}(1) = 0$.
	\end{itemize} Then $C^{c+m}_\tau \le C^c_\tau$.
\end{lem}

\begin{proof}
	We consider the two possible scenarios seen by the customer arriving at time $\tau$.
	\begin{enumerate}
		\item $1\in\mathcal A^{c+m}_{\tau-}$.
		\item $1\notin\mathcal A^{c+m}_{\tau-}$ and $1\notin\mathcal A^c_{\tau-}$.
	\end{enumerate}
	(Note that the third possibility, that $1\notin\mathcal
	A^{c+m}_{\tau-}$ and $1\in\mathcal A^{c}_{\tau-}$, is excluded by our
	assumption. Indeed, if $1\in\mathcal A^{c}_{\tau-}$ then our assumption
	forces $L^c_{\tau-}(1) = U^c_{\tau-}(1) = 0$. So the arrival at time $\tau$
	would find a server empty in $U^c$, and hence must
	also find a server empty in $U^{c+m}$ and, therefore, in $L^{c+m}$. But that would imply
	that $1\in\mathcal A^{c+m}_{\tau-}$.)
	
	We treat these two scenarios in order.
	\begin{enumerate}
		\item Since $1\in\mathcal A^{c+m}_{\tau-}$, we know from \eqref{eqn:change_in_C} that the
		coalescence time for $L^{c+m}$ and $U^{c+m}$ is unchanged by the new
		arrival. In addition, the coalescence time for $L^c$ and $U^c$
		cannot decrease due to this arrival. So
		\[ C^{c+m}_\tau = C^{c+m}_{\tau-} \le C^c_{\tau-} \le C^c_\tau \,. \]
		\item
		Here the arrival potentially affects the time until coalescence for
		both pairs of sandwiching processes. However,
		\[ C^{c+m}_\tau = \max\{C^{c+m}_{\tau-},\, 
		U^{c+m}_{\tau-}(1)+S \} \le \max\{C^c_{\tau-},\,  U^c_{\tau-}(1)+S
		\} = C^c_\tau \,, \]
		where the inequality follows from the second assumption of the Lemma, and the previously established fact that $U^{c+m}_{\tau-}\preceq U^c_{\tau-}$. 
	\end{enumerate}
\end{proof}

We can now complete the proof of Theorem~\ref{thm:main}. Recall that the sandwiching processes $L^{c+m}$ and $U^{c+m}$ are started at time $T<0$ with $L^{c+m}_T$ empty and $U^{c+m}_T$ instantiated using the same set of residual workloads that are present in $U^c_T$. Now, it is clear that departures in $U^{c+m}$ occur no later than in $U^c$, and since $C^c_T$ is simply the time taken for all customers present in $U^c_T$ to depart, it follows that 
\[ C^{c+m}_T \le C^c_T \,. \] 
Given that the assumption of Theorem~\ref{thm:main} holds, Lemma~\ref{lem:monotone} tells us that this ordering is preserved for
all $t\in[T,T^c]$:
\[ C^{c+m}_t \le C^c_t \,,\quad t\in[T,T^c] \,. \] But since $L^c$ and $U^c$ coalesce at time $T^c<0$, we see that $C^{c+m}_{T^c} = C^c_{T^c} = 0$, and so $T^{c+m}\le T^c$, as
claimed.

\section{Simulations}\label{sec:sims}

The result of Theorem~\ref{thm:main} provides us with a recipe for performing omnithermal perfect simulation for $M/G/(c+m)$ queues, for any $m\in\mathbb{N}\cup\{\infty\}$.
\vfill\eject

\begin{algorithm}\caption*{\textbf{Omnithermal Algorithm}}
	\label{algO}
	\begin{algorithmic}[1]
		\STATE Use Algorithm~\ref{alg2} to establish a backoff time $T<0$ and upper and lower sandwiching processes $U^c$ and $L^c$ over $[T,0]$ such that $U^c_0=L^c_0$:
		\begin{enumerate}
			\item[(i)] Calculate the coalescence time $T^c\in[T,0]$ of $U^c$ and $L^c$;
			\item[(ii)] If either 
			\begin{enumerate}
				\item[(a)] the condition of Theorem~\ref{thm:main} is satisfied for all arrival times in $[T,T^c]$, or
				\item[(b)] the upper process $U^c$ is empty at some time in $[T,0]$
			\end{enumerate}
			go to Step 2.
			\item[(iii)] Otherwise, set $T\leftarrow 2T$, and use Algorithm~\ref{alg2} to extend the simulation of the sandwiching processes over the new window $[T,0]$. \\Go back to Step 1(i).
		\end{enumerate}
		\STATE For each required $m\in\mathbb{N}\cup\{\infty\}$, construct $L^{c+m}$ over $[T,0]$, using the same sequence of arrival times and services as in the construction of $L^c$. Return $L^{c+m}_0$ as a perfect equilibrium draw of the Kiefer-Wolfowitz vector for the $M/G/(c+m)$ queue of interest.
	\end{algorithmic}
\end{algorithm}

\begin{rems}
	\begin{enumerate}
		\item 
		In Step 1(ii), in addition to checking whether the condition of Theorem~\ref{thm:main} is satisfied, we check whether the upper sandwiching process $U^c$ has emptied. This is another sufficient condition for coalescence of all pairs of sandwiching processes with more servers (as pointed out at the start of Section~\ref{sec:omnithermal}), and including this condition allows for a simple argument that the run-time of the Omnitheral Algorithm has finite expectation. 
		Indeed, let $T_{RA} = \sup\{t\le 0\,:\, Y_t = 0\}$, where $Y$ is the random assignment dominating process used in Algorithm~\ref{alg2} during Step 1. Our standing assumption that $\Expect{S^2}<\infty$ ensures that the stationary process $Y$ is positive recurrent, and so $\Expect{|T_{RA}|}<\infty$. If Step 1 uses a backoff time $T$ satisfying $T<T_{RA}$ then the upper process $U^c$ constructed over $[T,0]$ will clearly be empty at time $T_{RA}$ (since $U^c$ is dominated by $Y$); condition 1(ii)(b) of the Omnithermal Algorithm will then prevent any further backing off. Hence the final backoff time $T$, and hence the run-time of the Omnithermal Algorithm, has finite expectation as claimed.
		
		\item
		If we are called upon to use Step 1(iii) of the algorithm and extend the simulations of $U^c$ and $L^c$ further into the past, we are guaranteed that these new sandwiching processes ($\tilde U^c$ and $\tilde L^c$, say) will still coalesce by time $T^c$: this follows from Theorem 5.1 of \cite{Connor-Kendall-2015}, which implies that 
		\[ L^c_t \preceq \tilde L^c_t \preceq \tilde U^c_t \preceq U^c_t, \quad t\in[T,0]\,.  \]
		
		\item	
		
		Note that in Step 2 we have included the possibility that $m=\infty$: with infinitely many servers each customer is assigned its own server upon arrival, irrespective of how many customers are being served at the time. If the condition of Theorem~\ref{thm:main} holds then it is simple to see that the bounding processes $U^\infty$ and $L^\infty$ will coalesce before $T^c$, and thus the Omnithermal Algorithm can also be applied in this setting.
	\end{enumerate}
\end{rems}

Simulation results from 1,000 runs of the Omnithermal Algorithm for $M/M/c$ queues with various parameter combinations are presented in Table~\ref{fig:backoff}. In each case we recorded how many runs of Algorithm~\ref{alg2} required additional backoff (as in Step 1(iii) of the Omnithermal Algorithm) in order to produce an omnithermal sample. This increased with $\rho$, as might be expected. Possibly more surprising however, is the observation that for any fixed value of $\rho$ the proportion of runs which required extending initially increased before decreasing as a function of $c$. For those runs which did need extending, we also recorded the number of additional backoffs required. Table~\ref{fig:backoff} shows the median, upper quartile and maximum of these values: note that for all combinations of parameters the upper quartile was at most 2. This indicates that the additional computational overhead of using the Omnithermal Algorithm is relatively minimal in most cases considered here. 

The final line of each entry in Table~\ref{fig:backoff} shows the percentage of runs for which the stronger condition of Proposition~\ref{prop:L_never_full} was satisfied, i.e. for which no customer arriving at the lower sandwiching process $L^c$ before coalescence had to wait to commence service. Note the negative correlation between this figure and the proportion of runs which needed to be extended. For relatively low values of $\rho$ the stronger condition is nearly always satisfied, but for higher values the difference in practice between the conditions of Proposition~\ref{prop:L_never_full} and Theorem~\ref{thm:main} becomes much more apparent.

\newcommand{\sbc}[3]{\begin{tabular}{@{}c@{}}#1\%\\(#2)\\#3\%\end{tabular}}
\begin{table}[h]
	\caption{Simulation results obtained from applying the Omnithermal Algorithm to an $M/M/c$ queue with $\mu=1$ and a range of values of $c$ and $\rho=\lambda/(\mu c)$; 1,000 runs were performed for each combination of parameters. For each table entry: the first line shows the percentage of runs which needed extending further into the past using the binary backoff scheme in Step 1(iii) of the algorithm; for those runs which required extending, the second line reports the (median, upper quartile, maximum) of the number of additional backoffs required; the third line shows the percentage of runs for which the condition of Proposition~\ref{prop:L_never_full} was satisfied.}  \label{fig:backoff}
	\centering
	\begin{tabular}{c|c|c|c|c|c|c|c}
		$\rho\;\;\backslash \;\; c$	 & 2 & 4 & 8 & 16 & 32 & 64  \\ 
		\hline
		0.65 & \sbc{1.5}{1,2,3}{96}  & \sbc{3.2}{1,1,2}{93} & \sbc{7.6}{1,1,3}{94} & \sbc{1.5}{1,1,2}{98} & \sbc{0.4}{1,1,1}{99} & \sbc{0}{--}{100}  \\ 
		\hline
		0.75 & \sbc{2.3}{1,1,2}{88}  & \sbc{8.3}{1,1,2}{78}  & \sbc{18.5}{1,1,5}{76} & \sbc{13.9}{1,1,3}{83} & \sbc{6.2}{1,1,5}{92} & \sbc{0.7}{1,1,1}{99}  \\ 
		\hline
		0.85 & \sbc{1.9}{1,1,3}{69} & \sbc{14.2}{1,1,2}{47} & \sbc{32.3}{1,2,5}{30} & \sbc{34.4}{1,2,9}{24} & \sbc{30.6}{1,2,10}{26} & \sbc{14.3}{1,1,6}{53}
	\end{tabular}
\end{table} 

Finally, as a simple demonstration of the desirability of being able to produce omnithermal samples, we used our algorithm to investigate the effect on workload of changing server number for a heavily loaded $M/M/c$ queue. We ran the algorithm 5,000 times using arrival rate $\lambda =2.85$, service rate $\mu=1$ and $c=3$ ($\rho=0.95$); 333 runs (7\%) needed extending further into the past as in Step 1(iii) of the algorithm, with only two of these requiring more than two additional backoffs. We then used the output to produce perfect samples of the  Kiefer-Wolfowitz workload vectors for $m\in\{0,1,2,3\}$. Figure~\ref{fig:means} shows the mean value of each coordinate of the vectors obtained. Increasing the number of servers from three to four can be seen to decrease the value of the first coordinate (which represents the expected waiting time of a customer arriving in equilibrium) by a factor of ten. Further detail is provided in Figure~\ref{fig:CDFs}, where we show the effect on the distribution function of the remaining workload in equilibrium at the first and last coordinates of the Kiefer-Wolfowitz vectors for the same set of simulations. 

\begin{figure}[h]
	\centering
	\includegraphics[width=8cm]{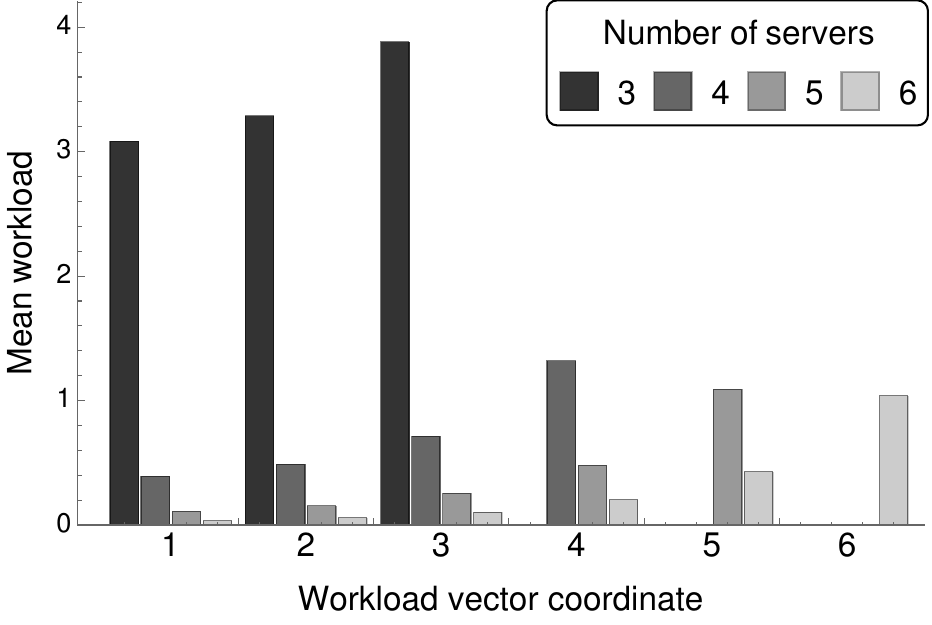}
	\caption{Mean of each coordinate of the workload vector for an $M/M/(c+m)$ queue with $\lambda = 2.85$, $\mu=1$, $c=3$ and $m\in\{0,1,2,3\}$. (Results from 5,000 runs of the Omnithermal Algorithm.) 
	}
	\label{fig:means}
\end{figure}

\begin{figure}[h]
	\centering
	{\includegraphics[width=6cm]{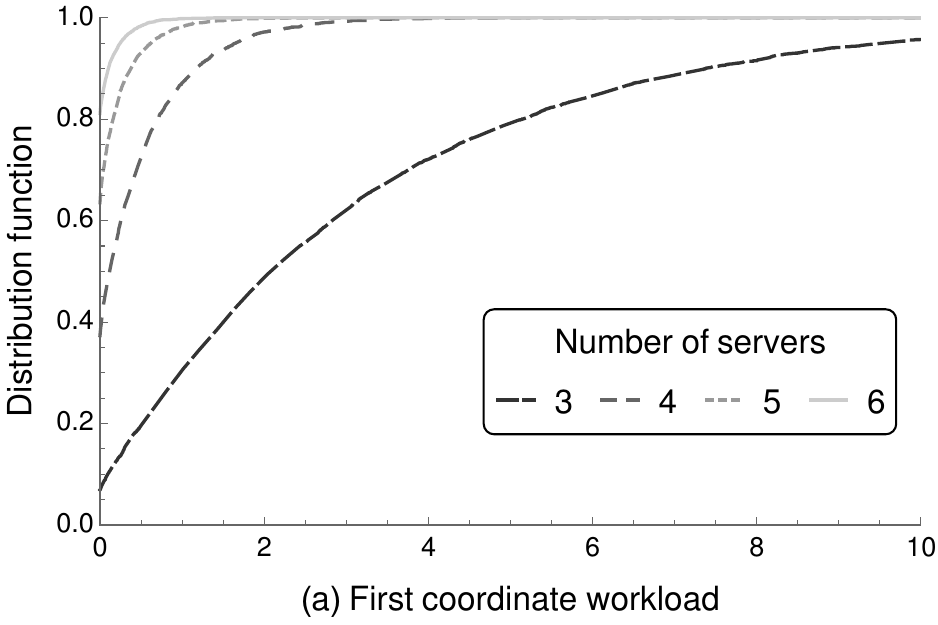}}\hspace{10mm}
	{\includegraphics[width=6cm]{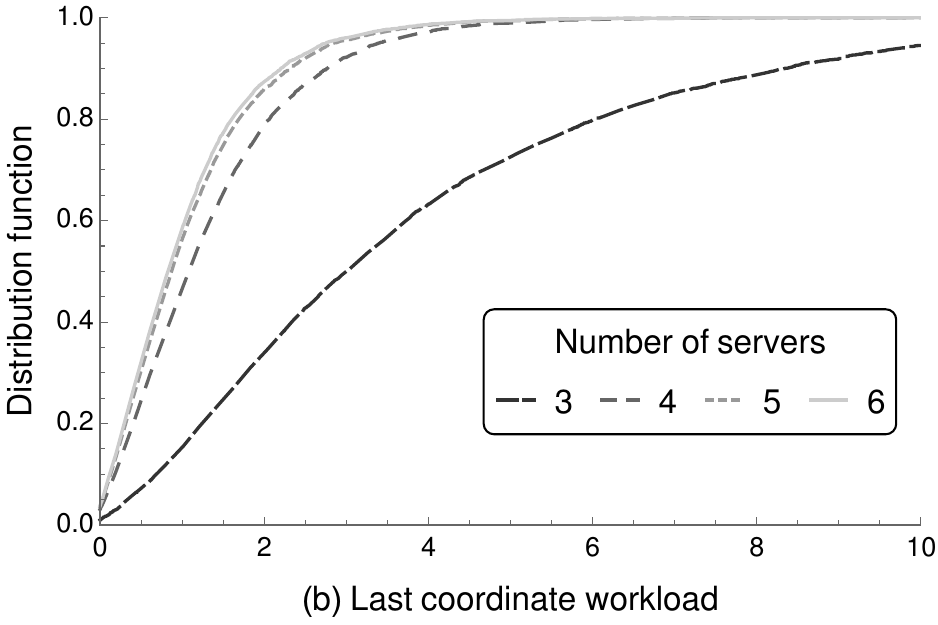}}
	\caption{Distribution functions for workload at (a) first and (b) last coordinates of the workload vector, for the set of simulations presented in Figure~\ref{fig:means}.
	}
	\label{fig:CDFs}
\end{figure}

\section{Variants and conclusions}\label{sec:conclusions}

We have shown how the efficient Algorithm~\ref{alg2} of \cite{Connor-Kendall-2015} for $M/G/c$ queues may be modified to allow for omnithermal perfect simulation; our new algorithm uses a simple test to determine whether or not the dominating process used for the $c$-server algorithm needs to be extended further into the past in order to allow for simultaneous sampling from $M/G/(c+m)$ queues for any $m\in\mathbb{N}\cup\{\infty\}$. The Omnithermal Algorithm has finite expected run-time and, furthermore, we have provided numerical evidence which suggests that for a wide range of parameters it involves relatively little additional computational expense. We conclude by briefly considering two variants of our algorithm. 

\subsection{Varying other system parameters}

An alternative natural setting in which one may be interested in omnithermal simulation is that in which the stability of the queue is increased by having shorter service durations, rather than adding more servers. (We could equivalently consider queues with longer inter-arrival times of course; however, for the domination arguments of Section~\ref{sec:omnithermal} to hold it is essential that the two systems being compared have the same set of arrival times. It is therefore more convenient to adjust the service durations instead.) Suppose that service times in the more stable system are distributed as $\beta S$ for some $\beta\in(0,1]$: this is equivalent to the service times being distributed as $S$, but with each server now completing work at rate $\beta^{-1}$. So we can compare the two systems as in Section~\ref{sec:omnithermal}, feeding both the same sets of arrival times and service durations, but with the time until coalescence in \eqref{eqn:C_defn} replaced by
\[ C_t^\beta = \beta U_t^\beta(n_t^\beta) \]
(where we have once again used the superscript to indicate the parameter that varies between the queues under consideration).

In a similar manner to Example~\ref{ex:fails}, it is easy to conjure up a set of arrival times and service durations such that the system completing work at rate $\beta^{-1}$ has $C_t^\beta>C_t^1$ for some values of $t$. However, we note that in this new setting equation~\eqref{eqn:change_in_C} becomes
\begin{equation*}
C^\beta_\tau = \begin{cases}
C^\beta_{\tau-} & \quad\text{if $1\in \mathcal A^\beta_{\tau-}$}\\
\max\{C^\beta_{\tau-},\, \beta(U^\beta_{\tau-}(1)+S) \} &\quad\text{if $1\notin \mathcal A^\beta_{\tau-}$}\,.
\end{cases}
\end{equation*}
Using this, it is a simple exercise to check that if the condition of Theorem~\ref{thm:main} is satisfied,  the sandwiching processes for the faster-working system will coalesce no later than do $U^c$ and $L^c$ for the original $M/G/c$ queue. In other words, we can perform omnithermal simulation in this setting by simply replacing Step 2 of the Omnithermal Algorithm with the following variant: 

\begin{itemize}
	\item[$2'$.] For any $\beta\in(0,1)$, construct $L^\beta$ over $[T,0]$, using the same set of arrival times and services as in the construction of $L^c$. Return $L^\beta_0$ as a perfect equilibrium draw of the Kiefer-Wolfowitz vector for the $M/G/c$ queue, in which work is completed at rate $\beta^{-1}$. 
\end{itemize}
(Step 1 -- in which we possibly extend some simulations further into the past -- does not change at all.)

In addition, we note that the coalescence arguments underpinning Section~\ref{sec:omnithermal} do not rely in any way on the distribution of inter-arrival times. As noted in the introduction, \citet{Blanchet-etal-2015} have recently shown how to implement domCFTP for $GI/GI/c$ queues using a random assignment dominating process with upper and lower sandwiching processes in the style of Algorithm~\ref{alg2} above. It is therefore possible to perform omnithermal simulation for these queues, by using their algorithm in place of Algorithm~\ref{alg2} in Step 1 of the Omnithermal Algorithm.

\subsection{Perfect simulation for adaptive systems}
In practical queueing situations it may be possible, indeed desirable, for a queue manager to alter the number of servers being employed at any given time, in response to either endogenous or exogenous effects, in order to strike a balance between server utilization and customer waiting times. A variety of mathematical models exist for such adaptive systems, with relevance to applications in telecommunication and road traffic networks. See, for example, \cite{LI2000615,Kafetzakis-2011,Bruneel-2016}.

Given that the Omnithermal Algorithm allows for simultaneous sampling of $M/G/(c+m)$ queues for any $m\ge 0$, it is natural to wonder whether it can also be applied to systems in which the parameter $m$ is allowed to vary as a function of the set of customers present in the system. Unfortunately, for many natural models of adaptive systems it is not the case that the monotonicity of workload vectors established in Proposition~\ref{prop:dom_different_c} is guaranteed to be maintained; in particular, it becomes possible for customers to depart from the upper sandwiching process sooner than from the lower one. Examples of such models include ones in which the number of servers at time $t$, $c_t$, evolves as a function of $c_{t-}$ and either the length of time since some server was last idle (the length of the current busy period), or  the number of customers waiting to begin service at time $t$.

Similarly, models in which servers can take vacations when idle, or in which the service rate can be altered as a function of the number of customers waiting, can be seen to exhibit monotonicity problems. The only sensible adaptive model which seems to (somewhat obviously) maintain monotonicity between upper and lower sandwiching processes is one in which $c_t$ depends upon $\mathrm{sgn}(Q_t-c_{t-})$, where $Q_t$ is the number of customers in the system at time $t$. That is, $c_t$ depends upon whether (taking into account a possible arrival or departure at time $t$) the system has a surplus, just the right number, or a deficit of servers. (But, importantly, the \emph{size} of any deficit can't be used to control $c_t$.) With this setup, one could allow $c_t$ to evolve according to the rule: \emph{if there is a deficit, and there are servers to spare, add one immediately; if there is a surplus, reduce the number of servers if you wish}. This system still obeys the monotonicity of Proposition~\ref{prop:dom_different_c},  meaning it is possible to sample perfectly from its equilibrium distribution using a simple variant of the Omnithermal Algorithm. However, the resulting equilibrium could just as easily be obtained by sampling from the $M/G/(c+m_{\max})$ system (where $(c+m_{\max})$ is the maximum available number of servers, possibly equal to $\infty$) and then ignoring any servers which are idle at time zero.

For more interesting adaptive systems, for which the monotonicity of Proposition~\ref{prop:dom_different_c} fails to hold, it may be possible to carry out perfect simulation under the assumption that the (variable) number of servers used is always at least $c$, with the corresponding $M/G/c$ system being stable. In this case we may be able to use the Omnithermal Algorithm with Step 1(ii)(a) removed: for reasonable control processes (e.g. ones in which $c_t$ is increasing in some measure of how busy the system is), $U^c$ will dominate the workload vector for the adaptive system started from zero at time $T$, and so we simply have to backoff until $U^c$ empties before time zero. Variations on this idea may be necessary, or of course there may be some other monotonicity which could be exploited to produce an entirely different algorithm, depending upon the exact control policy. Note that the time taken for $U^c$ to empty will in general be rather large (e.g. for the simulations presented in Table~\ref{fig:backoff}, $U^c$ emptied before time zero in less than $1\%$ of runs with $c\ge 8$), but will certainly be no greater than the run-time of Algorithm~\ref{alg1}.


\bibliographystyle{chicago} \bibliography{PerfectQueues}

\end{document}